\documentclass[11pt]{amsart}
\usepackage{amsmath,amscd,amssymb}
\usepackage[mathscr]{eucal}
\usepackage[all]{xy}

 \newtheorem{thm}{Theorem}[section]
 \newtheorem{cor}[thm]{Corollary}
 
 \newtheorem{prop}[thm]{Proposition}

 \theoremstyle{definition}

 \theoremstyle{remark}
 \newtheorem{rem}[thm]{Remark}

\newtheorem{example}[thm]{Example}

\newtheorem{remark}[thm]{Remark}

\numberwithin{equation}{section}
\newtheorem{theorem}{Theorem}[section]
\newtheorem{proposition}[theorem]{Proposition}

\newcommand{\vanish}[1]{\relax}

\newcommand{\N}{\mathbb{N}}

\newcommand{\ud}{\mathrm{d}}

\newcommand{\Sum}[2][\relax]{%
 \ifx#1\relax \sideset{}{_{#2}}\sum
 \else \sideset{}{^{#1}_{#2}}\sum
 \fi}

\newcommand{\Ce}{\mathrm{C}}

\DeclareMathOperator{\dom}{dom}
\DeclareMathOperator{\ran}{ran}

\newcommand{\norm}[2][\relax]{%
   \ifx#1\relax \ensuremath{\left\Vert#2\right\Vert}
   \else \ensuremath{\left\Vert#2\right\Vert_{#1}}
   \fi}

\makeatletter
\newcommand{\sprod}[2]{\ensuremath{%
  \setbox0=\hbox{\ensuremath{#2}}
  \dimen@\ht0
  \advance\dimen@ by \dp0
  \left(\left.#1\rule[-\dp0]{0pt}{\dimen@}\,\right|#2\hspace{1pt}\right)}}
 \makeatother

\newcounter{aufzi}

\newcounter{aufzii}

\newcounter{aufziii}

\begin{document}

\title[
Generation of subordinated holomorphic semigroups
 ]
{Generation of subordinated holomorphic semigroups via Yosida's
theorem}

\author{Alexander Gomilko}
\address{Faculty of Mathematics and Computer Science\\
Nicolas Copernicus University\\
ul. Chopina 12/18\\
87-100 Toru\'n, Poland  and  Institute of Telecommunications and
Global Information Space
\\
National Academy of Sciences of Ukraine\\
 13, Chokolivsky Blvd\\
Kiev, Ukraine 03186}

\email{gomilko@mat.umk.pl}

\author{Yuri Tomilov}
\address{
Institute of Mathematics\\
Polish Academy of Sciences\\
\'Sniadeckich 8\\
00-956 Warszawa, Poland }

\email{tomilov@mat.umk.pl}

\thanks{This work was completed with  a partial support by the NCN grant
 DEC-2011/03/B/ST1/00407 and by the EU Marie Curie IRSES program, project "AOS", No. 318910.}

\dedicatory{To Charles Batty, colleague and friend, on the
occasion \\
of his sixtieth anniversary with admiration}



\date{\today}

\begin{abstract}
Using functional calculi theory, we obtain several estimates for
$\|\psi(A)g(A)\|$, where $\psi$ is a Bernstein function, $g$
is a bounded completely monotone function and  $-A$ is the
generator of a holomorphic $C_0$-semigroup on a Banach space,
bounded on $[0,\infty)$. Such estimates are of value, in particular, in approximation theory of operator semigroups.
As a corollary, we obtain a new proof of
the fact that $-\psi(A)$ generates a holomorphic semigroup
whenever $-A$ does,  established recently in \cite{GT} by a different approach.
\end{abstract}

\date{\today}
\maketitle

\section{Introduction}

Bernstein functions  play an important role in analysis, and in
particular, in the study of L\'evy processes in probability
theory. Recently they found a number of applications in operator
and ergodic theories, mainly in issues related to rates of
convergence of semigroups and related operator families. At a core
of many applications of Bernstein functions is an abstract subordination
principle going back to Bochner, Nelson and Phillips (see \cite[p. 171]{SchilSonVon2010} for more on its historical background).
%Roughly, it allows
%one to define a Bernstein function $\psi$ of a  generator $-A$ of
%a bounded $C_0$-semigroup. The resulting operator $-\psi(A)$ is
%the generator of a bounded $C_0$-semigroup again. Moreover, the
%Bernstein function links the two $C_0$-semigroups (see below for
%more on that).
Given a Bernstein function $\psi$ and a generator $-A$ of a bounded $C_0$-semigroup on a Banach space $X$,
the principle allows one to define the operator $-\psi(A)$ which again is the generator of a bounded $C_0$-semigroup
on $X.$
Thus, it is natural to ask whether Bernstein
functions preserve other classes of (bounded) semigroups relevant
for applications such as holomorphic, differentiable or any of
their subclasses. This paper treats the permanence of the class of
holomorphic semigroups under Bernstein functions.

Recall that a positive  function $g\in
\Ce^\infty\bigl(0,\infty\bigr)$ is called {\em completely monotone} if
\[
(-1)^n  g^{(n)}(\tau) \geq 0, \qquad \tau >0,
\]
for each $n\in\N$.

A positive function $\psi\in \Ce^\infty(0,\infty)$ is called a
{\em Bernstein function} if its derivative is completely monotone.

A basic property of  Bernstein functions is that their
exponentials arise as Laplace transforms of a uniquely defined
convolution semigroups of subprobability measures. This property
is a core of the notion of subordination discussed below.

Recall that a family of Radon measures $(\mu_t)_{t\ge 0}$ on $[0,\infty)$ is called a vaguely continuous convolution
semigroup of subprobability measures if for all $t\ge 0, s \ge 0,$
\[
\mu_t([0,\infty))\le 1, \quad \mu_{t+s}=\mu_t*\mu_s, \quad
\text{and} \quad \mbox{vague}-\lim_{t\to0+}\mu_t=\delta_0,
\]
where $\delta_0$ stands for the Dirac measure at zero. Such a
semigroup is often called a subordinator. The next classical
characterization of Bernstein functions goes back to Bochner and
can be found e.g. in \cite[Theorem 5.2] {SchilSonVon2010}.
\begin{thm}\label{Bochner}
A function $\psi:(0,\infty)\to (0,\infty)$ is Bernstein if and
only if there exists a vaguely continuous convolution semigroup
$(\mu_t)_{t \ge 0}$ of subprobability measures on $[0,\infty)$
such that
\begin{equation}\label{CMonG}
\widehat \mu_t(\tau):=\int_0^\infty e^{- s \tau}\,\mu_t(ds)=e^{-t\psi(\tau)},\quad \tau > 0,
\end{equation}
for all $t \ge 0.$
\end{thm}

Theorem \ref{Bochner} has its operator-theoretical counterpart.
One of most natural ways to construct a new $C_0$-semigroup from a
given one is to use subordinators. Recall that if $(e^{-tA})_{t
\ge 0}$ is  a bounded $C_0$-semigroup on a Banach space $X$ and
$(\mu_t)_{t \ge 0}$ is a vaguely continuous convolution semigroup
of bounded Radon measures on $[0,\infty)$ then the formula
\begin{equation}
e^{-t\mathcal A}=\int_0^\infty e^{-s A}\,\mu_t(ds), \qquad t\ge 0,
\end{equation}
defines a bounded $C_0$-semigroup $(e^{-t\mathcal A})_{t \ge 0}$
on $X$ whose generator $-\mathcal A$ can be considered as $-\psi
(A)$, thus we will write $\psi(A)$ instead of $\mathcal A$ (see
the next subsection for more on that). The $C_0$-semigroup
$(e^{-t\psi(A)})_{t \ge 0}$ is called subordinated to the
$C_0$-semigroup $(e^{-tA})_{t \ge 0}$ via the subordinator
$(\mu_t)_{t \ge 0}$ (or the corresponding Bernstein function
$\psi$).

Despite the construction of subordination is very natural and
appears often in various contexts, some of its permanence
properties have not been made precise so far. In this note, we
show the permanence of semigroup holomorphicity of under
subordination.
%For the proofs of other permanence properties, e.g.
%preservation of holomorphy angles, see \cite{GT}.
In particular, we present a positive answer to  the following open
question posed in \cite[p. 63]{Rob}, see also \cite{Laub}: suppose
that $-A$ generates a bounded holomorphic $C_0$-semigroup on a
Banach space $X$ and $\psi$ is a Bernstein function. Does
$-\psi(A)$ also generate a (bounded) holomorphic $C_0$-semigroup ?

A partial answer to a strengthened version of this question  was
given in \cite[Proposition 7.4]{Laub}: for any Bernstein function
$\psi$ the operator $-\psi(A)$ generates a sectorially bounded
holomorphic $C_0$-semigroup of angle at least $\theta$ if $-A$
generates a sectorially bounded holomorphic $C_0$-semigroup of
angle $\theta$ greater then $\pi/4$. Moreover, it was proved in
\cite[Theorem 6.1 and Remark 6.2]{Laub} that the above claim is
true with no restrictions on $\theta \in (0,\pi/2]$ if a Bernstein
function $\psi$ is, in addition, \emph{complete}. It was asked in
\cite{Laub} whether this additional assumption can, in fact, be
removed.

If  $X$ is a uniformly convex Banach space, e.g. if $X$ is a
Hilbert space, then a positive answer to Kishimoto-Robinson's
question was obtained in \cite[Theorem 1]{MirS} using Kato-Pazy's
criteria for semigroup holomorphicity.

Recently, based on the machinery of functional calculi, positive
answers to both questions in their full generality, were provided
in  \cite{GT}. In particular, it was proved in \cite{GT} that  if
$-A$ generates a sectorially bounded holomorphic $C_0$-semigroup
of  angle $\theta$, then for any Bernstein function $\psi$ the
operator $-\psi(A)$ also generate a sectorially bounded
holomorphic $C_0$-semigroup of angle at least $\theta$.

The aim of this note is to present an alternative and
comparatively simple argument providing positive answers to the
questions from \cite{Rob} and \cite{Laub} apart from the
permanence of holomorphy angles property. (This property requires
additional arguments going beyond the scope of the paper, see
\cite{GT} for its proof.) Our approach has merits of being
self-contained, transparent and much less technical in a sense of
using only elementary properties of functional calculi theory.

The proof arises as a byproduct of  estimates for
$\|\psi(A)e^{-t\varphi(A)}\|$, $t>0$, where $\psi$, $\varphi$ are
Bernstein functions, satisfying appropriate conditions. In turn
such estimates appeared to be crucial in putting approximation
theory of operator semigroups into the framework of Bernstein
functions of semigroup generators, see \cite{GTJFA}. In fact, the
techniques developed in \cite{GTJFA} is basic in this paper.

It is not clear whether the permanence of semigroup holomorphy
sectors can be proved by the methods of present note. See however
\cite{BGTo} where still another, direct approach to subordination
was worked out in details.

\section{Preliminaries}
\subsection{Function theory}

Let us recall some basic facts on completely monotone and
Bersntein functions from \cite{SchilSonVon2010} relevant for the
following.

First, note that by Bernstein's theorem \cite[Theorem 1.4]{SchilSonVon2010} a real-valued function $g\in
\Ce^\infty(0,\infty)$ is completely monotone if and only if it is
the Laplace transform of a (necessarily unique) positive
Laplace-transformable Radon measure $\nu$ on $[0,\infty):$
\begin{equation}\label{bernst}
g(\tau) = \widehat \nu(\tau)=\int_{0}^{\infty} e^{-\tau s}\nu(ds)
\qquad \mbox{for all}\quad \tau > 0.
\end{equation}
In particular, \eqref{bernst} implies that a completely monotone
function extends holomorphically to the open right half-plane
$\mathbb C_+:=\{z\in \mathbb C: {\rm Re}\, z > 0\}.$ The set of
complete monotone functions will be denoted by $\mathcal{CM}$, and
the set ot bounded complete monotone functions will be denoted by
$\mathcal{BCM}$. The standard examples of completely monotone
functions include $e^{-t\tau}, \tau^{-\alpha},$ for fixed $t >0$
and $\alpha \in [0,1],$ and $(\log(1+\tau))^{-1}.$

Bernstein functions constitute a class ``dual'' in a sense to the
class of completely monotone functions.
A representation similar
in a sense to \eqref{bernst} holds also for Bernstein functions.
Indeed, by \cite[Thm. 3.2]{SchilSonVon2010}, a function $\psi$
is a Bernstein function  if and only if there exist  $a, b\geq 0$
and a positive Radon measure $\gamma$ on $(0,\infty)$ satisfying
\begin{equation*}
\int_{0+}^\infty\frac{s}{1+s}\,\gamma(\ud{s})<\infty \label{mu}
\end{equation*}
such that
\begin{equation}\label{predGA}
\psi(\tau)=a+b\tau+\int_{0+}^\infty
(1-e^{-s\tau})\gamma(d{s}),\qquad \tau>0.
\end{equation}
The formula \eqref{predGA} is called the L\'evy-Khintchine
representation of $\psi.$ The triple $(a, b, \gamma)$ is uniquely
determined by $\psi$ and is called the L\'evy-Khintchine triple.
Thus we will write occasionally $\psi \sim (a,b, \gamma).$
 Note that a Bernstein function
$\psi \sim (a,b, \gamma)$ is increasing,  and it satisfies
\[
a=\psi(0+)\quad \text{and} \quad b=\lim_{t \to \infty}\frac{\psi(t)}{t}.
\]
Moreover, by \eqref{predGA}, $\psi$ extends holomorphically to
$\mathbb C_+$ and continuously to the closure $\overline{\mathbb
C}_+.$ The Bernstein function $\psi$ is bounded  if and only if
$b=0$ and $\gamma ((0,\infty))<\infty$, see \cite[Corollary
3.7]{SchilSonVon2010}.

In the sequel, we will denote the of set of Bernstein functions by
$\mathcal{BF}$. As  examples of Bernstein functions we mention
$1-e^{-t\tau},$ $\tau^\alpha,$ for fixed $t >0$ and $\alpha \in
[0,1],$ and $\log (1+\tau).$

Now we introduce a functional $J$ which will be an important tool
in getting operator norm estimates for the products of functions
of a negative semigroup generator $A.$

For $g\in \mathcal{CM}$
and $\psi\in \mathcal{BF}$ let us define
\begin{equation}\label{JAT}
J[g, \psi]:=\int_0^\infty g(s)\psi'(s)\,ds.
\end{equation}
Note that $J$ is  well-defined if we allow $J[g, \psi]$ to be $\infty.$

The following choice of $g$ and $\psi$ will be of particular importance.
%Recall that a composition of a completely monotone function and a Bernstein function is completely monotone (see \cite{SchilSonVon2010}).
%Thus,
Observe that if   $t>0$ is fixed, $\varphi$ is a Bernstein
function, and $g=e^{-t\varphi}$  then $g \in \mathcal{BCM}$ by
Theorem \ref{Bochner} and
\begin{equation}\label{GGJ}
J[e^{-t\varphi}, \psi]=\int_0^\infty e^{-t\varphi(s)}\psi'(s)\,ds.
\end{equation}

Let us note several conditions on $g$ and $\psi$  guaranteeing that $J[g, \psi]$ is finite.

\begin{example}\label{Mon}
\,$a)$\,Let $g \in \mathcal{CM}$ and $\psi\in \mathcal{BF}$. If
there exists a continuous function $q: (0,\infty)\mapsto
(0,\infty)$ such that
\begin{equation}\label{ppCT0}
\int_0^\infty q(s)\,ds<\infty,\quad \text{and}\quad  g(s)\le
q(\psi(s)),\qquad s>0,
\end{equation}
then
\[
J[g, \psi]\le \int_0^\infty q(\psi(s)) \psi'(s)\,ds
=\int_{\psi(0)}^{\psi(\infty)} q(s)\,ds \le \int_0^\infty
q(s)\,ds<\infty.
\]

On the other hand, if  $g\in \mathcal{CM}$, $\psi\in \mathcal{BF}$
and $J[g,\psi]<\infty,$ then
\[
g(\tau)=q(\psi(\tau)),\quad z>0,\quad q(s):=g(\psi^{-1}(s)),\quad
s\in (\psi(0),\psi(\infty)),
\]
and
\begin{eqnarray*}
\int_{\psi(0)}^{\psi(\infty)} q(s)\,ds
&=&\int_{\psi(0)}^{\psi(\infty)} g(\psi^{-1}(s))\,dt
\\
&=&\int_0^\infty g(\psi^{-1}(\psi(s))\,d\psi(s) =\int_0^\infty
g(s)\psi'(s)\,ds<\infty.
\end{eqnarray*}
Thus, \eqref{ppCT0} is also necessary (in a sense described above)
for $J[g, \psi]<\infty.$

$b)$\, Let $g \in \mathcal {BCM}$ be such that $g(0)\le 1$ and
$g(\infty)=0,$  and let $\psi\in \mathcal{BF}$. Suppose that
there exists a continuous function $f: (0,1)\mapsto (0,\infty)$
such that
\begin{equation}\label{ppCT1}
\int_0^1 f(s)\,ds<\infty,\quad \text{and} \quad \psi(s)\le
f(g(s)),\qquad s>0.
\end{equation}
Then
\begin{equation}\label{JLem}
J[g, \psi]\le \int_0^1 f(s)\,ds.
\end{equation}
Indeed, note that $g'(s)<0,$ $s < 0.$ Then, by (\ref{ppCT1}),  for
all $\epsilon>0$ and $\tau>1,$
\begin{eqnarray}
\int_\epsilon^{\tau} g(s) \psi'(s)\,ds
&=& g(\tau) \psi(\tau)- g(\epsilon) \psi(\epsilon)-\int_\epsilon^{\tau}  g'(s)\psi(s)\,ds \label{limACT}\\
&\le&  g(\tau) f(g(\tau))-\int_\epsilon^{\tau}  g'(s)f(g(s))\,ds\notag \\
&=& g(\tau) f(g(\tau))+\int_{g(\tau)}^{g(\epsilon)} f(s)\,ds\notag\\
&\le& g(\tau) f(g(\tau))+\int_0^1 f(s)\,ds.\notag
\end{eqnarray}
Note that $g(\tau)$ decrease to zero monotonically as $\tau \to
\infty.$ Since $f\in L^1(0,1)$ there exists $(\tau_k)_{k\ge
1}\subset (1,\infty)$ such that
\[
 \lim_{k\to\infty}\,\tau_k=\infty,\quad \text{and}\quad
\lim_{k\to\infty}\,g(\tau_k)f(g(\tau_k))=0.
\]
Since $g$ and $\psi'$ are positive,  setting $\tau=\tau_k, k\in
\N,$ in (\ref{limACT}) and letting $k\to\infty$ and $\epsilon\to
0$, we obtain (\ref{JLem}).
\end{example}

We proceed with several estimates for $J[g,\psi],$ where $g$ is of
the form $e^{-t\varphi}, t >0,$ for a Bernstein function
$\varphi$. They will be important for exploring holomorphicity of
$(e^{-t\varphi(A)})_{ t \ge 0}$ in the next section.

\begin{example}\label{ExT}\,
$a)$\,For any $\psi\in \mathcal{BF},$ we have
\begin{eqnarray}
J[e^{-t\psi},\psi]&=&\int_0^\infty e^{-t\psi(s)}\psi'(s)\,ds
\label{ref3}\\
&=&t^{-1}[e^{-t\psi(0)}-e^{-t\psi(\infty)}]\le t^{-1},\qquad
t>0.\notag
\end{eqnarray}

$b)$\, If $\psi\in \mathcal{BF}$ and
$\varphi_\alpha(\tau):=\tau^\alpha$, $\alpha\in (0,1]$, then using
monotonicity of $\psi$ and the fact that
\begin{equation}\label{max}
\psi(c\tau)\le c\psi(\tau),\qquad \tau\ge 0,\quad c \ge 1,
\end{equation}
 see e.g. \cite[p. 205]{Jacob}, it follows that
\begin{eqnarray}\label{q4}
J[e^{-t\varphi_\alpha},\psi]+\psi(0)&=& t\alpha\int_0^\infty e^{-ts^\alpha}s^{\alpha-1}\psi(s)\,ds\\
&=&\alpha \int_0^\infty e^{-s^\alpha} s^{\alpha-1} \psi(s/t^{1/\alpha})\,ds \notag \\
&\le& \psi(1/t^\alpha)\int_0^\infty
e^{-s}\max \{1,s^{1/\alpha}\}\,ds \notag \\
&\le& \left(1+\frac{1}{\alpha
e}\right)\,\psi(1/t^\alpha),\qquad t>0.\notag
\end{eqnarray}
Let now $\psi \sim (a,b,\gamma)$ and $\alpha=1$ so that
$\varphi_1(\tau)=\tau.$ Then using (\ref{predGA}), the inequality
\[
\frac{s}{t+s}=\frac{s/t}{1+s/t}\le 1-e^{-s/t},\qquad s,t>0,
\]
and Fubini's theorem, we infer that
\begin{eqnarray}\label{q5}
J[e^{-t\varphi_1},\psi]&=& \int_0^\infty e^{-tz}\psi'(s)\,ds
=\frac{b}{t}+\int_{0+}^\infty \frac{s}{t+s}\,\gamma(ds)
\\
&\le& \frac{b}{t}+ \int_{0+}^\infty
(1-e^{-s/t})\,\gamma(ds)=\psi(1/t)-\psi(0)\notag \\
&\le& \psi(1/t),\qquad t>0. \notag
\end{eqnarray}

The following estimate for $J$ generalizes the one in a).

$c)$\, Let $\psi$ be a bounded Bernstein  function satisfying
\begin{equation}\label{CGG}
\psi(0)=0,\qquad \psi'(0+)<\infty,
\end{equation}
and let $\varphi$ be a Bernstein function. Then,
\[
J[e^{-t\varphi},\psi]=\int_0^\infty
e^{-t\varphi(s)}\psi'(s)\,ds\le \psi(\infty),\qquad t>0.
\]
On the other hand, if we are interested in asymptotics of
$J[e^{-t\varphi},\psi]$ for big $t$ and $\varphi\not\equiv
\mbox{const}$, then a better estimate is available. Since
\[
\varphi(\tau)=\int_0^\tau \varphi'(s)\,ds+\varphi(0)\ge\varphi'(1)
\tau,\qquad \tau\in (0,1),
\]
it follows that
\begin{eqnarray*}
J[e^{-t\varphi},\psi]&=&\int_0^1
e^{-t\varphi(s)}\psi'(s)\,ds+\int_1^\infty e^{-t\varphi(s)}\psi'(s)\,ds\\
&\le& \psi'(0)\int_0^1
e^{-t\varphi'(1)s}\,ds+e^{-t\varphi(1)}\int_1^\infty \psi'(s)\,ds\\
&\le& \left[\frac{\psi'(0)}{\varphi'(1)}+
\frac{\psi(\infty)-\psi(1)}{\varphi(1)}\right]\frac{1}{t},\qquad t>0.
\end{eqnarray*}
\end{example}

We finish this subsection with several estimates shading a light
on interplay between the functional $J[g,\psi]$ and the product
$g\cdot \psi.$ They will be needed as an illustration of our main
statement.

The following estimate is well-known for so-called complete
Bernstein functions. However, it seems, it has not been noted for
the whole class of Bernstein functions. In the proof, we use an
idea from the proof of  \cite[Theorem 4]{Carasso}.

\begin{proposition}\label{EstB}
Let $\psi\in \mathcal{BF}$. Then
\begin{equation}\label{Eb1}
|\psi(z)|\le 2\sigma^{-1}\varphi(|z|),\quad {\rm Re}\,z\ge 0,\quad \sigma=1-e^{-1}.
\end{equation}
\end{proposition}

\begin{proof}
Recall that
\[
|1-e^{-z}|\le \min (|z|,2)\le 2\min (|z|,1),\quad {\rm Re}\,z\ge 0,
\]
and
\[
1-e^{-s}\ge \sigma \min (s,1),\quad s\ge 0,\quad \sigma=1-e^{-1},
\]
see  \cite[Lemma 2.1.2]{Jacob}. Therefore,
\begin{equation}\label{sigma}
|1-e^{-z}|\le 2\sigma^{-1}(1-e^{-|z|}),\quad {\rm Re}\,z\ge 0.
\end{equation}
Let $\psi\in \mathcal{BF}$ be given by (\ref{predGA}). Then,
using (\ref{sigma}) and noting that $1<2\sigma^{-1}$, we obtain
\begin{eqnarray*}
|\psi(z)|&\le& a +b|z|+\int_{0+}^\infty |1-e^{-sz}|\,\gamma(ds)\\
&\le& a +b|z|+2\sigma^{-1}\int_{0+}^\infty
(1-e^{-|z|s})\,\gamma(ds)\\
&\le& 2\sigma^{-1}\psi(|z|),\qquad {\rm Re}\,z\ge 0.
\end{eqnarray*}
\end{proof}

In the following result, we show that for  $g \in \mathcal {BCM}$
and $\psi\in\mathcal{BF}$ the assumption $J[g, \psi]<\infty$
implies that $g\cdot\psi$ is bounded in any sector
\[
\Sigma_\beta:=\{z\in \mathbb C: |\arg z|<\beta \},\qquad \beta\in (0,\pi/2).
\]

\begin{cor}\label{CorLaub10}
Let  $\psi\in\mathcal{BF}$. Then the
following statements hold.
\begin{itemize}
\item [(i)]  For every $g \in \mathcal {CM}$ and every $\beta\in (0,\pi/2),$
\begin{equation}\label{ResCA}
|g(z)\psi(z)|\le \frac{2}{\sigma \cos\gamma}\,
g(|z|\cos\beta)\psi(|z|\cos\beta), \quad  z\in \Sigma_\beta.
\end{equation}
\item [(ii)] Let $g \in \mathcal {BCM}$ and $J[g, \psi]<\infty$.  Then for every $\beta\in (0,\pi/2),$
\begin{equation}\label{ResC}
|g(z)\psi(z)|\le
\frac{2}{\sigma\cos\beta}\,\{g(0+)\psi(0)+J[g,\psi]\}, \qquad z\in
\Sigma_\beta.
\end{equation}
\end{itemize}
\end{cor}

\begin{proof}
To prove (i) suppose that $g$ is given by (\ref{bernst}) and $z\in \Sigma_\beta.$ Then
\[
|g(z)|\le \int_0^\infty e^{-s{\rm Re}\,z}\,\nu(ds) \le
\int_0^\infty e^{-s|z|\cos\beta}\,\nu(ds) =g(|z|\cos\beta).
\]
Using Proposition \ref{EstB} and the inequality
\eqref{max}, we then obtain
\[
|g(z)\psi(z)|\le 2\sigma^{-1}g(|z|\cos\beta)\psi(|z|) \le
\frac{2}{\sigma \cos\beta}\, g(|z|\cos\beta)\psi(|z|\cos\beta),
\]
and the proof is complete.

If $J[g, \psi]<\infty,$ then, since $g$ is decreasing, for every
$\tau>0,$
\begin{eqnarray*}
g(\tau)\psi(\tau)&=&g(0+)\psi(0)+\int_0^\tau [g'(s)\psi(s)+g(s)\psi'(s)]\,ds\\
&\le& g(0+)\psi(0)+\int_0^\tau g(s)\psi'(s)]\,ds\le
g(0+)\psi(0)+J[g, \psi].
\end{eqnarray*}
Hence, by (i),
\[
|g(z)\varphi(z)|\le \frac{2}{\sigma
\cos\beta}\,\{g(0+)\phi(0)+J[g,\psi]\},\qquad z\in \Sigma_\beta,
\]
so that (ii) holds.
\end{proof}

\subsection{Functional calculus and holomorphic semigroups}

In this subsection we will set up the extended Hille-Phillips
functional calculus. The calculus will enable us to define
Bernstein functions of a negative semigroup generator and to
establish some of their basic properties including operator
counterparts of the formulas \eqref{CMonG} and \eqref{predGA}. As
we will see below, the formulas remain essentially the same upon
replacement an independent variable by an operator $A.$

Let ${\rm M_b}(\mathbb R_+)$ be a Banach algebra of bounded Radon
measures on $\mathbb R_+:=[0,\infty)$  with the standard, total
variation norm $\|\mu\|_{{\rm M_b}(\mathbb R_+)}:=|\mu|(\mathbb
R_+).$ Note that
\[
{\rm A}^1_+(\mathbb C_+) := \{ \widehat \mu : \mu \in {\rm
M}_b(\mathbb R_+)\}
\]
is also a commutative Banach algebra with pointwise multiplication
and with the norm inherited from ${\rm A}^1_+(\mathbb C_+):$
\begin{equation}\label{mmm}
\|\rm \hat \mu\|_{{\rm A}^1_+(\mathbb C_+)} := \|\mu\|_{{\rm
M_b}(\mathbb R_+)}.
\end{equation}

Let $-A$ be the generator of a bounded $C_0$-semigroup
$(e^{-tA})_{t\ge 0}$ on $X$. Define an  algebra homomorphism
$\Phi: {\rm A}^1_+(\mathbb C_+) \mapsto {\mathcal L}(X)$ by the
formula
\[ \Phi (\widehat {\mu})x := \int_0^{\infty} e^{-sA}x\, \mu(d{s}),
\qquad x \in
 X.
\]
Since
\begin{equation}\label{hillestimate}
 \| \Phi (\hat {\mu})\| \le \sup_{t \ge 0} \|e^{-tA}\| |\mu|(\mathbb R_+),
\end{equation}
$\Phi$ is clearly continuous.
The homomorphism $\Phi$ is called
the {\em Hille-Phillips} (HP-) functional calculus for $A$. If $g
\in {\rm A}^1_+(\mathbb C_+)$ so that $g=\widehat \mu$ for $\mu
\in {\rm M_b}(\mathbb R_+),$ we then  put
\begin{equation*}
g(A)= \Phi(\widehat {\mu}).
\end{equation*}
Basic properties of the Hille-Phillips functional calculus can be
found in \cite[Chapter XV]{HilPhi} and in \cite[Chapter 3.3]{Haa2006}. It is
crucial to note that if $g \in \mathcal {BCM}$, then $g\in
A^1_+(\mathbb C_+)$ by Fatou's theorem, so that $g(A)$ is defined
in the HP-calculus and $g(A) \in \mathcal L(X).$

 Let now $O(\mathbb C_+)$ be an algebra of functions holomorphic in $\mathbb C_+.$
 Denote by ${\rm A}^{1}_{+,r} (\mathbb C_+)$ the set of $f \in O(\mathbb C_+)$ such that there exists
 $e\in {\rm A}^1_{+}(\mathbb C_+)$ with $ef \in {\rm A}^1_+(\mathbb C_+)$ and the
operator $e(A)$ is injective. Then  for any $f \in {\rm
A}^{1}_{+,r} (\mathbb C_+)$ one defines $f(A)$ as
\begin{eqnarray}\label{general}
f(A):=(e(A))^{-1} ef(A).
\end{eqnarray}
The above definition does not depend on the choice of a
regularizer $e$, and thus the mapping $f \to f(A)$ is
well-defined. We will call this mapping the {\em extended
Hille--Phillips calculus} for $A$.

The extended HP-calculus satisfies, in particular, the following,
natural sum and product rules, see e.g. \cite[Chapter 1]{Haa2006}.

\begin{prop}\label{fcrules}
Let $f$ and $g$ belong to ${\rm A}^{1}_{+,r} (\mathbb C_+)$, and let $-A$ be the generator of a bounded $C_0$-semigroup.
Then
\begin{itemize}
\item [(i)] $f(A)g(A) \subset (fg)(A);$
\item [(ii)] $f(A)+g(A)\subset (f+g)(A);$
\end{itemize}
If $g(A)$ is bounded then the inclusions above are, in fact,
equalities.
\end{prop}

Recall  that, as it was shown in \cite[Lemma 2.5]{GHT}, Bernstein
functions are regularisable by $e(z)=1/(1+z),$ that is $e\psi \in
{\rm A}^1_+(\mathbb C)$ for every Bernstein function $\psi,$ and
then, in particular, by the HP-calculus,
\begin{equation}\label{fcinclusion}
[\psi(z) (1+z)^{-1}](A) \in \mathcal L (X).
\end{equation}
Thus, according to \eqref{general}, for any $\psi \in
\mathcal{BF},$
\begin{equation}\label{part}
\psi(A)=(1+A)[\psi(z) (1+z)^{-1}](A).
\end{equation}

While Bernstein functions can  formally be defined in the extended
$HP$-calculus by \eqref{general}, this definition can hardly be
used for practical purposes. However, following analogy to the
scalar-valued case, one can derive representations for operator
Bernstein functions similar to  \eqref{CMonG} and \eqref{predGA},
see e.g. \cite[Corollary 2.6]{GHT} and \cite[Proposition 2.1 and
Theorem 12.6] {SchilSonVon2010}.

\begin{thm}\label{BochP}
Let $-A$ generate a bounded $C_0$-semigroup $(e^{-tA})_{t \ge 0}$
on  $X$, and let $\psi$ be a Bernstein function with the
corresponding L\'evy-Hintchine triple $(a, b,\gamma).$  Then
the following statements hold.
\begin{itemize}
\item [(i)]  For every  $x\in \dom(A),$
 \begin{equation}\label{clos}
\psi(A)x = ax + bAx + \int_{0+}^\infty (1 - e^{-sA})x \,\gamma(d{s}),
\end{equation}
where the integral is understood as a Bochner integral. Moreover,
${\dom}(A)$ is  core for $\psi (A).$
\item [(ii)] The operator $-\psi(A)$ generates a bounded $C_0$-semigroup $(e^{-t\psi(A)})_{t \ge 0}$
on $X$ given by
\begin{equation}\label{entt}
e^{-t\psi(A)}:=\int_{0}^{\infty} e^{-sA}\, \mu_t(ds), \qquad t \ge
0,
\end{equation}
where $(\mu_t)_{t\ge 0}$ is a vaguely continuous convolution
semigroup   of subprobability measures on $[0,\infty)$
corresponding to $\psi$ by \eqref{CMonG}.
\end{itemize}
\end{thm}

Thus, the operator Bernstein function $\psi(A)$ can be recovered
from its restriction to $\dom (A)$ by means of \eqref{clos}.
Moreover, $-\psi(A)$ generates a bounded $C_0$-semigroup if $-A$
does, and this fact motivates further study of the permanence
properties for the mapping $-A \to -\psi(A),$ e.g. preservation of
the class of generators of holomorphic semigroups on $X$.

It will be crucial to note that subordination does not increase
the norm. Indeed, as an immediate consequence of Theorem
\ref{BochP}, (ii), one obtains
\begin{equation}\label{SupG}
 \sup_{t\ge 0}\,\|e^{-t\psi(A)}\|\le \sup_{\ge >0}\,\|e^{-t A}\|.
\end{equation}

While the relations \eqref{clos} and \eqref{entt}  hold for any
bounded $C_0$-semigroup, in this note we will concentrate on
bounded $C_0$-semigroups which are, in addition, holomorphic.
Recall that a  $C_0$-semigroup $(e^{-tA})_{t \ge 0}$ is said to be
holomorphic   if it extends holomorphically  to a sector
$\Sigma_\beta$ for some $\beta \in (0,\frac{\pi}{2}]$ which is
bounded on $\Sigma_{\theta}\cap\{z\in \mathbb C: |z|\le 1\}$ for
any $\theta\in (0,\beta)$. If $e^{-\cdot A}$ is bounded in
$\Sigma_\theta$ whenever $0<\theta<\beta,$ then $(e^{-tA})_{t \ge
0}$ is said to be a sectorially bounded holomorphic semigroup of
angle $\beta.$

It is well-known that sectorially bounded holomorphic semigroups
can be described by means of their asymptotics on the real axis.
Namely,  $-A$ is the generator of a sectorially bounded
holomorphic $C_0$-semigroup $(e^{-tA})_{t \ge 0}$ on a Banach
space $X$ if and only if $e^{-tA}(X) \subset \dom(A)$ for every $t
>0,$ and $\sup_{t\ge 0}\,\|e^{-tA}\|$ and
$\sup_{t>0}\,\|tAe^{-tA}\|$ are finite, see e.g. \cite[Theorem
4.6]{EngNag}.

It is often useful to omit the assumption of sectorial boundedness
and to consider  $C_0$-semigroups bounded on $\mathbb R_+$ and
having a holomorphic extension to a sector around the real axis.
This situation can also be characterized in terms of behavior of
$(e^{-tA})_{t \ge 0}$ on the positive half-axis.

By a classical Yosida's theorem, a $C_0$-semigroup
$(e^{-tA})_{t\ge 0}$ on $X$ is holomorphic if and only if
%$(e^{-tA})_{t\ge 0}$ satisfies
\begin{equation}\label{Y}
e^{-tA}(X)\subset \dom(A), \quad t>0, \, \, \text{and} \,\,
\limsup_{t\to 0}\,\|tAe^{-tA}\|<\infty.
\end{equation}
Since it is no easy to find this statement  in the literature, we
sketch its proof below. Note that by \cite[Proposition 3.7.2 $
b)$]{ABHN01} a $C_0$-semigroup $(e^{-tA})_{t\ge 0}$ on $X$ is
holomorphic if and only if there exists $a>0$ such that
$(e^{-t(A+a)})_{t\ge 0}$ is a sectorially bounded holomorphic
$C_0$-semigroup. Then, by \cite[Theorem 4.6]{EngNag} mentioned
above, the latter property is equivalent to $e^{-tA}(X) \subset
\dom(A)$ for every $t
>0,$ and
\begin{equation}\label{bound}
\sup_{t>0}\, (e^{-at}\|e^{-tA}\|+\|te^{-at}Ae^{-tA}\|)<\infty.
\end{equation}
Thus, in particular, \eqref{Y} holds. Conversely, if \eqref{Y} is
true, then \eqref{bound} is satisfied for certain $a>0,$ and the
sectorial boundedness of $(e^{-t(A+a)})_{t\ge 0}$ yields the
holomorphicity of $(e^{-tA})_{t\ge 0}.$ (Concerning Yosida's
theorem and its proof see also \cite{Yosida} and \cite[Remark, p.
332]{Kom}.)

Note that if $(e^{-tA})_{t\ge 0}$ is holomorphic and bounded, then
for all $\delta >0$ and $t > \delta,$
\[
\|Ae^{-tA}\|\le  \left(\sup_{t \ge 0 }\|e^{-tA}\|\right)
\sup_{t\in (\delta/2,\delta)} \|Ae^{-t A}\|.
\]
In other words,  if $(e^{-tA})_{t\ge 0}$ is bounded, then the Yosida condition (\ref{Y}) can be given the
equivalent form
\begin{equation}\label{MMGS11}
\|Ae^{-tA}\|\le c_0+\frac{c_1}{t},\qquad t>0,
\end{equation}
with some constants $c_0\ge 0$ and $c_1>0$ which will be crucial
in the estimates below. Thus, if $(e^{-tA})_{t \ge 0}$ satisfies
\eqref{MMGS11}, then we say that $(e^{-tA})_{t \ge 0}$ satisfies
\emph{the Yosida condition $Y(c_0,c_1)$ }(which is just an
explicit form of the classical Yosida condition \eqref{Y} above).

It will be convenient to rewrite \eqref{MMGS11} in terms of only
$(e^{-tA})_{t \ge 0}.$ To this aim, we first prove the following
simple proposition.

\begin{prop}\label{estS}
Let $-A$ be the generator of a bounded  $C_0$-semigroup on a
Banach space $X$ such that
\begin{equation}\label{MMGS}
\sup_{t\ge 0}\,\|e^{-tA}\|\le M.
\end{equation}
Suppose that
$e^{-tA}(X) \subset \dom(A)$, $t>0$, and there exists
an increasing function $r :\,(0,\infty)\mapsto
(0,\infty)$ such that
\begin{equation}\label{MMGS1}
\sup_{t>0}\,r(t)\|Ae^{-tA}\|\le 1.
\end{equation}
Then
\begin{equation}\label{maxG}
\|(1-e^{-s A})\,e^{-t A}\|\le
\frac{4M\, s}{2Mr(t)+s},\qquad s, t>0.
\end{equation}
\end{prop}

\begin{proof}
By (\ref{MMGS}), for all $s, t>0,$
\[
\|(1-e^{-s A})\,e^{-t A}\|\le 2M.
\]
On the other hand, since
\begin{equation}\label{integral}
(1-e^{-s A})\,e^{- tA}= \int_t^{t+s} Ae^{-\tau A} \, d\tau,
\end{equation}
we infer  by (\ref{MMGS1}) that
\begin{eqnarray*}
\|(1-e^{-s A})\,e^{- tA}\|&\le&
 \le
 \int_t^{t+s} \frac{d\tau}{r(\tau)}
 \le
\frac{s}{r(t)},\qquad s,t>0.
\end{eqnarray*}
Then, since
\[
\min\,\{a, b\}\le \frac{2ab}{a+b},\qquad a,b>0,
\]
it follows that
\[
\|(1-e^{-s A})\,e^{-t A}\|\le
\min\,\{2M,s/r(t)\}\le \frac{4M\, s}{2Mr(t)+s}.
\]
\end{proof}

Now we are ready to recast \eqref{MMGS11} in semigroup terms, and
the following corollary of Proposition \ref{estS} is almost
immediate.
\begin{cor}\label{CorR}
Let $-A$ be the generator of a $C_0$-semigroup on  $X$
satisfying (\ref{MMGS}) and the Yosida condition $Y(c_0, c_1).$
Then
\begin{equation}\label{maxG1}
\|(1-e^{-s A})\,e^{-t A}\|\le 2s\left\{\frac{2Mc_0}{1+c_0s}+
\frac{\max(2M,c_1)}{t+s}\right\},\qquad s, t>0.
\end{equation}
Conversely, if estimate \eqref{maxG1} holds, then $(e^{-tA})_{t\ge 0}$ satisfies
the Yosida condition $Y(4M c_0, 2\max(2M,c_1)).$
\end{cor}

\begin{proof}
By  Proposition \ref{estS} applied to
\[
r(t):=\frac{t}{c_0t+c_1},\qquad t>0,
\]
we obtain that
\begin{eqnarray*}
\|(1-e^{-s A})\,e^{-t A}\|&\le&
4Ms\frac{(c_0t+1)}{2Mt+(c_0t+c_1)s}\\
&=&4Ms\left\{\frac{c_0t}{2Mt+(c_0t+c_1)s}+
\frac{c_1}{2Mt+(c_0t+c_1)s}\right\}\\
&\le& 4Ms\left\{\frac{c_0}{2M+c_0s}+ \frac{c_1}{2Mt+c_1s}\right\}\\
&\le& 2s\left\{\frac{2Mc_0}{1+c_0s}+ \frac{\max(2M,c_1)}{t+s}\right\}.
\end{eqnarray*}
If, conversely, \eqref{maxG1} is true, then dividing both sides of it by
$s$, using \eqref{integral} and passing to the limit as $s \to 0+$ for a fixed $t >0,$
we get
\begin{equation*}
\|Ae^{-tA}\| \le 4Mc_0+ 2\frac{\max(2M,c_1)}{t},
\end{equation*}
that is  $Y(4M c_0, 2\max(2M,c_1))$ holds.
\end{proof}

The elementary estimate \eqref{maxG1} will play a key role in the subsequent arguments.

\section{Main results}

To obtain a positive answer to Kishimoto-Robinson's question, we
need to show that if   $(e^{-tA})_{t\ge 0}$ is a bounded $C_0$-semigroup
satisfying Yosida's condition, then for any Bernstein function $\psi$ one has
$e^{-t\psi(A)}(X) \subset {\rm dom}\, (\psi(A)), t >0,$ and the
function $t \mapsto  \|t \psi(A)e^{-t\psi(A)}\|$ is bounded
in an appropriate neighborhood of zero. This will be derived as a simple
consequence of the following operator norm estimate for
 $\psi(A)g(A)$ where  $\psi\in\mathcal{BF}$ and  $g \in \mathcal{BCM}.$
In a different context, a related estimate was obtained in \cite[Theorem 1]{Mir4}.

For the rest of the paper, if $(e^{-tA})_{t \ge 0}$ is a bounded $C_0$-semigroup on a Banach space $X$
then we let
$$
M(A):=\sup_{t \ge 0}\|e^{-tA}\|.
$$

\begin{thm}\label{Laub1}
Let $\psi\in\mathcal{BF}$ and  $g \in \mathcal{BCM}$ be such
that $J[g, \psi]<\infty$. Let $-A$ be
the generator of a bounded $C_0$-semigroup satisfying
the Yosida condition $Y(c_0, c_1).$
Then $\psi(A)g(A) \in \mathcal L(X)$ and
\begin{eqnarray}\label{App1}
\|\psi(A)g(A)\|&\le& \psi(0)\|g(A)\|\\
&+& 2\max(M(A),c_1) J[g, \psi]+4M(A)g(0+)C[c_0;\psi],\notag
\end{eqnarray}
where
\begin{equation}\label{Const1}
C[c_0;\psi]:= \int_0^\infty e^{-s/c_0}\psi'(s)\,ds, \qquad c_0>0,\qquad C[0;\psi]:=0.
\end{equation}
\end{thm}

\begin{proof}
By assumption and Bernstein's theorem, there exists a finite
Radon measure $\nu$ on $[0,\infty)$ such that
\begin{equation}\label{sim}
g(s)=\int_0^\infty e^{- \tau s}\,\nu(d\tau),\quad s > 0,\qquad
g(0+)=\nu([0,\infty))<\infty.
\end{equation}
Let $\varphi\sim (a,b,\gamma)$ so that the representation
(\ref{predGA}) holds.
Then \eqref{Const1} takes the form
\begin{equation*}
C[c_0;\psi]=bc_0+\int_{0+}^\infty
\frac{c_0 s}{1+c_0 s}\,\gamma(ds).
\end{equation*}

Note that it suffices to prove (\ref{App1}) for a Bernstein
function $\psi$ with $a=\psi(0)=0$.

Suppose first that  $a=b=0$ in (\ref{predGA}). Let $x\in
\dom(A)$ be fixed. Then, by
\eqref{fcinclusion} and \eqref{fcrules},
\[
g(A)x\in \dom(A)\subset \dom \psi(A).
\]
Hence, by Fubini's theorem, we have
\begin{eqnarray*}
\psi(A)g(A)x=g(A)\psi(A)x&=&
\int_0^\infty e^{-\tau A}\,\nu(d\tau)\int_{0+}^\infty [1-e^{-s A}]x\,\gamma(ds)\\
&=&\int_0^\infty  \int_{0+}^\infty [1-e^{-s A}]e^{-\tau A}x\,
\gamma(ds)\,\nu(d\tau).
\end{eqnarray*}
Using  (\ref{maxG1}) and (\ref{sim}), from here it follows that
\begin{eqnarray}\label{estim}
&&\|\psi(A)g(A)x\|\\
&\le& 2\|x\|\int_0^\infty \int_{0+}^\infty
\left\{\frac{2M(A)c_0 s}{1+c_0s}+\frac{\max(2M,c_1)s}{\tau+s}\right\}\,\gamma(ds)\,\nu(d\tau)
\notag \\
&=&
2\|x\|\left\{g(0)C[c_0;\psi]
+\max(2M(A),c_1) \int_0^\infty\int_{0+}^\infty
\frac{s}{\tau+s}\,\gamma(ds)\,\nu(d\tau)\right\}.\notag
\end{eqnarray}
Again, by applying Fubini's theorem twice, we obtain that (as in
\eqref{q5})
\begin{equation*}\label{GDer}
\int_0^\infty e^{-\tau t} \psi'(t)\,dt= \int_{0+}^\infty
\frac{s\,\gamma(ds)}{s+\tau},\quad \tau>0.
\end{equation*}
and
\begin{eqnarray}\label{JJP}
\int_0^\infty \int_{0+}^\infty
\frac{s\,\gamma(ds)}{s+\tau}\,\nu(d\tau)&=& \int_0^\infty
\int_0^\infty e^{-\tau t}\psi'(t)\,dt\,\nu(d\tau)
\\
&=&\int_0^\infty \int_0^\infty e^{-\tau t} \,\nu(d\tau)\,\psi'(t)\,dt\notag \\
&=&\int_0^\infty g(t)\psi'(t)\,dt\notag \\
&=&J[g, \psi].\notag
\end{eqnarray}
So, \eqref{estim} yields
\begin{equation}\label{domGA1}
\|\psi(A)g(A)x\|\le
2\|x\|\{\max(2M(A),c_1)J[g, \psi]+2M(A)g(0)C[c_0;\psi]\}.
\end{equation}

From (\ref{domGA1}), since $\psi(A)g(A)$ is closed as product
of closed and bounded operators and
$\dom(A)$ is dense in $X$, we conclude  that
\begin{equation}\label{Inc}
\ran(g(A))\subset \dom(\psi(A)),
\end{equation}
and \eqref{App1} holds.
This finishes the proof in the case $a=b=0.$

Let now $a=0$ and  $b>0$. Arguing as above, if $x\in \dom(A)$ is fixed, then
\begin{eqnarray*}
\psi(A)g(A)x&=&g(A)\psi(A)x
\\
&=& b\int_0^\infty Ae^{-\tau A}x\,\nu(d\tau)
+\int_0^\infty  \int_{0+}^\infty [1-e^{-s A}]\,
e^{-\tau A}x\,\gamma(ds)\,\nu(d\tau).
\end{eqnarray*}
Note that $\psi'(s)\ge b$, $s>0$, and
\begin{eqnarray*}
\int_0^\infty \tau^{-1}\nu(d\tau)&=&
\int_0^\infty g(s)\,ds\\
&\le& b^{-1}\int_0^\infty g(s)\psi'(s)\,ds
=b^{-1}J[g, \psi]<\infty.
\end{eqnarray*}
Therefore,
\begin{equation}\label{aaa}
\|A g(A)x\|\le \int_{0}^{\infty}\|Ae^{-\tau A}x\|\, \nu(d\tau)
\le \|x\|\int_{0}^{\infty}(c_0+c_1/\tau)\,\nu(d\tau).
\end{equation}

Now using (\ref{JJP}) for a  Bernstein function $\psi(t)-bt,$
and taking into account \eqref{aaa}, we obtain that
\begin{eqnarray*}
\|\psi(A)g(A)x\|&\le& b\|x\| \int_0^\infty
(c_0+c_1\tau^{-1})\,\nu(d\tau)
\\
&+&2\|x\|\int_0^\infty \int_{0+}^\infty
\left\{\frac{2 M(A)c_0 s}{1+c_0s}+\frac{\max (2M(A),c_1) s}{\tau+s}\right\}\,\gamma(ds)\,\nu(d\tau)\\
&\le& g(0+)bc_0\|x\| + \|x\|b \int_0^\infty g(s)\,ds
\\
&+&4M(A) g(0+)\|x\| \int_{0+}^\infty \frac{c_0 s}{1+c_0s}\,\gamma(ds)\\
&+&2\max(M(A),c_1)\|x\|\int_0^\infty g(s)[\psi'(s)-b]\,ds
\\
&\le& 4M(A)g(0+)\|x\|\left\{ b c_0 + \int_{0+}^\infty \frac{c_0
s}{1+c_0s}\,\gamma(ds)\right\}\\
&+&2\max(2M(A),c_1)\|x\|\int_0^\infty
g(s)\psi'(s)\,ds\\
\\
&=&2\|x\|\,\{\max(2M(A),c_1) J[g,\psi]+2M(A)g(0+)C[c_0;\psi]\}.
\end{eqnarray*}
Since the operator $\psi(A)g(A)$ is closed and  $\dom(A)$ is
dense,  the last inequality implies  (\ref{Inc}) and (\ref{App1}).
\end{proof}

\begin{rem}
The assumption $J[g,\psi]<\infty$ is not necessary to ensure the
boundedness of $\psi(A)g(A).$ To see this, it is enough to
consider a Bersntein function $\psi(\tau)=\tau+1$ and a bounded
completely monotone function $g(\tau)=1/(\tau+1).$ However, the
assumption implies the boundedness of $\psi\cdot g$ in any sector
$\Sigma_\beta$ with $\beta \in (0,\pi/2),$ see Corollary
\ref{CorLaub10}. If $-A$ generates a sectorially bounded
holomorphic $C_0$-semigroup and admits, in addition, a bounded
$H^{\infty}$-calculus on a sector $\Sigma_\theta$, the boundedness
of $\psi\cdot g$ in $\Sigma_\beta,$ $\beta>\theta,$ implies also
the boundedness of $\psi(A)g(A).$
\end{rem}

For a  choice of $g$ as $e^{-t\varphi},$ where $\varphi$ is a
Bernstein function,  Theorem \ref{Laub1} yields immediately the
following corollaries.

\begin{cor}\label{Laub2}
Let $\psi$ and $\varphi$ be  Bernstein  functions such that
$J[e^{-t\varphi},\psi]<\infty$ for every $t>0$. Let $-A$ be the
generator of a bounded $C_0$-semigroup on $X$ satisfying the Yosida condition $Y(c_0, c_1).$
Then for every $t >0,$
\begin{eqnarray*}
\|\psi(A)e^{-t\varphi(A)}\| &\le& \psi(0)\|e^{-t\varphi(A)}\|\label{ABC1G02}\\
&+& 2 \max(2M(A),c_1)
J[e^{-t\varphi},\psi]+4M(A)e^{-t\varphi(0)}C[c_0,\psi].
\end{eqnarray*}
\end{cor}

\begin{cor}\label{AunboundedG2}
Let $\psi$  be a Bernstein function  and let $-A$ be the
generator of a  bounded $C_0$-semigroup $(e^{-tA})_{t \ge 0}$ on
$X$ satisfying the Yosida condition $Y(c_0, c_1).$ Then for every $t >0,$
\begin{equation}\label{ABC1G4}
\| \psi(A)e^{-t A}\|\le
2\max(2M(A),c_1)\psi(1/t)+4M(A)C[c_0;\psi].
\end{equation}
In particular, if $-A$ generates a sectorially bounded holomorphic
$C_0$-semigroup, then
\begin{equation}\label{ABC1G}
\|\psi(A)e^{-t A}\|\le 2 \max(2M(A),c_1)\psi(1/t),\qquad t>0.
\end{equation}
\end{cor}

\begin{proof}
By (\ref{q5}) and Corollary  \ref{Laub2} applied to a Bernstein
function $\varphi_1(\tau)=\tau$,
\begin{eqnarray*}
\|\psi(A)e^{-t A}\|&\le&
\psi(0)M(A)+2\max(2M(A),c_1)J[e^{-t\varphi_1},\psi]+4M(A)C[c_0;\psi]\\
&\le& 2\max(2M(A),c_1)\{J[e^{-t\varphi_1},\psi]+\psi(0)\}+4M(A)C[c_0:\psi]\\
&\le& 2\max(2M(A),c_1)\psi(1/t)+4M(A)C[c_0;\psi].
\end{eqnarray*}
\end{proof}

As we explained in the beginning of this section, Corollary
\ref{Laub2} leads to a positive answer to Kishimoto-Robinson's
question which is contained in the next statement. Incidentally,
it also partially answers the question from \cite{Laub} and shows
that Bernstein functions map the class of generators of
sectorially bounded holomorphic $C_0$-semigroups into itself. The
statement was proved in \cite{GT} by a different technique.
\begin{cor}\label{Laub3}
Let $\psi$ be  a Bernstein  function and let $-A$ be the generator
of a bounded $C_0$-semigroup satisfying the Yosida condition $Y(c_0,c_1).$ Then
$-\psi(A)$ generates a bounded $C_0$-semigroup on $X$
satisfying the following Yosida condition:
\begin{equation}\label{ABC1G0}
\|\psi(A)e^{-t\psi(A)}\| \le
M(A)(\psi(0)+4)C[c_0;\psi])e^{-t\psi(0)}
%\\
+2\max(2M(A),c_1)t^{-1} % \notag
\end{equation}
for every $t>0.$
If $-A$ generates a sectorially bounded $C_0$-semigroup on $X,$
then the same is true for $-\psi(A).$
\end{cor}

\begin{proof}
Note that $\psi=\psi(0)+\psi_0$, $\psi_0\in
\mathcal{BF},$ and then
\begin{equation}\label{ref2}
\|e^{-t\psi(A)}\|\le e^{-\psi(0)t}\|e^{-t\psi_0(A)}\|\le M(A), \qquad t >0.
\end{equation}
Now  Corollary \ref{Laub2} and Example \ref{ExT}, $a)$ yield \eqref{ABC1G0}.
If $(e^{-tA})_{t \ge 0}$ is sectorially bounded, then $c_0=0$ and, by definition, $C[c_0;\psi]=0$ as well.
In this case, \eqref{ABC1G0} implies that $t\psi(A)e^{-t\psi(A)}$ is bounded on $(0,\infty).$
Since $(e^{-t\psi(A)})_{t \ge 0}$ is bounded, it is moreover sectorially bounded.
\end{proof}

Next we turn to other applications of Theorem \ref{Laub1} arising
in a general framework for approximation theory of operator
semigroups developed in \cite{GTJFA}. Note that Corollary
\ref{Laub2} and Example \ref{ExT}, \,$c)$  imply directly the next
statement (cf. \cite[Theorem 6.8]{GTJFA}).

\begin{thm}\label{AunboundedG}
Let $\psi$ be a bounded Bernstein  function satisfying
(\ref{CGG}), and let $\varphi\not\equiv \mbox{const}$ be a Bernstein
function. Let $-A$ be the generator of a  sectorially bounded holomorphic
$C_0$-semigroup $(e^{-tA})_{t \ge 0}$ on $X$.  Then
\begin{equation}\label{ABC1G00}
\sup_{t>0}\, \|t \psi(A)e^{-t\varphi(A)}\|\le
2\max(2M(A),c_1)\left[\frac{\psi'(0)}{\varphi'(1)}+\frac{\psi(\infty)-\psi(1)}{\varphi(1)}\right].
\end{equation}
\end{thm}

The following corollary of Theorem \ref{AunboundedG} was obtained
in \cite[Corollary 6.9]{GTJFA}.

\begin{cor}\label{C11G}
Let $\varphi$ be a Bernstein function such that
\begin{equation}\label{varphi}
\varphi'(0+)=1,\quad |\varphi''(0+)|<\infty.
\end{equation}
Let $-A$ be the generator of a  sectorially bounded holomorphic
$C_0$-semigroup $(e^{-tA})_{t \ge 0}$ on $X$. Then
\begin{equation*}\label{ABC1G11}
\|(1-\varphi'(A))e^{-t\varphi(A)}\|\le
\frac{2\max(2M(A),c_1)}{t}\left[\frac{|\varphi''(0+)|}{\varphi'(1)}+
\frac{\varphi'(1)}{\varphi(1)}\right],
\end{equation*}
for all $t >0$.
\end{cor}

\begin{proof}
Note that by \eqref{varphi} the Bernstein function
$\psi(\tau)=1-\varphi'(\tau), \tau>0,$ is bounded  and satisfies
\eqref{CGG}. Applying  Theorem \ref{AunboundedG} to a Bernstein
function $\varphi$ and a bounded Bernstein function $\psi$ and
taking into account the relations $
\psi'(0+)=-\varphi''(0+)=|\varphi''(0+)|$ and
\[
\psi(\infty)-\psi(1)=\varphi'(1)-\varphi'(\infty)\le \varphi'(1),
\]
we get the assertion.
\end{proof}

\begin{remark}
Note that in \cite[Theorem 6.8]{GTJFA} the second term
$\frac{\psi(\infty)-\psi(1)}{\varphi(1)}$ in the right hand of
\eqref{ABC1G00} has a wrong form $\psi(1)/\varphi(1)$ due to
incorrect evaluation of
$\|\psi'\|_{L^1([a,\infty))}=\int_{a}^{\infty}\psi'(s)\, ds$ in
the last line of the proof. Thus \cite[Eq. (6.12)]{GTJFA} should
take a form of \eqref{ABC1G00}. However, \cite [Corollary 6.9]
{GTJFA} (i.e. Corollary \ref{C11G} here) which was a base for
subsequent estimates in \cite[Section 6]{GTJFA} remains unchanged.
\end{remark}

We finish with relating our estimates to the following
generalization of the moment inequality for generators of bounded
$C_0$-semigroups given in  \cite[Corollary
12.18]{SchilSonVon2010}. As proved in \cite{SchilSonVon2010}, if
$-A$ is the generator of a bounded $C_0$-semigroup on $X$ and
$\psi\in \mathcal{BF}$, then
\begin{equation}\label{Shil}
\|\psi(A)x\|\le
\frac{2e}{e-1}M(A)\psi\left(\frac{\|Ax\|}{2\|x\|}\right),\qquad
x\not=0,\qquad x\in \dom(A).
\end{equation}
If $\psi(\tau)=\tau^\alpha, \alpha \in (0,1),$ then \eqref{Shil}
reduces to the classical moment inequality for fractional powers
of $A$. Using our technique, we obtain the following corollary of
\eqref{Shil}.

\begin{cor}\label{Ok}
Let $-A$ be the generator of a bounded $C_0$-semigroup such that
\begin{equation}\label{condS}
\|tAe^{-tA}\|\le {M}_a,\qquad t\in (0,a],\qquad a\le\infty,
\end{equation}
and $\psi\in \mathcal{BF}$.
Then
\begin{equation}\label{SS2}
\|\psi(A)e^{-tA}\|\le \frac{e}{e-1}M(A)
\max\{2M(A),{M}_a\}\,\psi(1/t),\qquad t\in (0,a].
\end{equation}
\end{cor}

\begin{proof}
Setting in (\ref{Shil})  $x=e^{-tA}y$, $y\in X$, $t\in (0,a]$ and
using (\ref{condS}) and  (\ref{max}), we obtain that
\begin{eqnarray*}
\|\psi(A)e^{-tA}y\|
&\le& \frac{2e}{e-1}M(A)\|e^{-tA}y\|\psi\left(\frac{M_a\|y\|}{2t\|e^{-tA}y\|}\right)\\
&\le&  \frac{2e}{e-1}M(A) \|e^{-tA}y\|\max\left\{1,\frac{M_a\|y\|}{2\|e^{-tA}y\|}\right\}\,\psi(1/t)\\
&=& \frac{e}{e-1}M(A) \max\left\{2\|e^{-tA}y\|,M_a\|y\|\right\}\,\psi(1/2t)\\
&\le& \frac{e}{e-1}M(A)\max\{2M(A),M_a\}\,\psi(1/t)\|y\|,
\end{eqnarray*}
that is (\ref{SS2}) holds.
\end{proof}

As an illustration of Corollary \ref{Ok}, note that if
$\psi(\tau)=\log(1+\tau)$ then Corollary \ref{Ok} yields the
estimate
\[
\sup_{t\in (0,1/e]}\,\frac{\|\log(1+A)e^{-tA}\|} {\log
(1/t)}<\infty.
\]
proved originally in \cite[Proposition 2.7]{Okazawa}.

Finally, we note that it is possible to develop an approach to the permanence problems from \cite{Rob} and \cite{Laub}
different from the ones in \cite{GT} and in the present note. This approach based on direct resolvent estimates for
Bernstein functions of semigroup generators is worked out in \cite{BGTo}. While it allows one to get sharp estimates
for subordinated semigroups  (and their holomorphy sectors), it is much more involved than the arguments in this article.

\end{document}